\documentclass{ourlematema}

\usepackage[hidelinks]{hyperref}
 \hypersetup{
    colorlinks,
    citecolor=magenta,
    filecolor=magenta,
    linkcolor=blue,
    urlcolor=black
}
\usepackage{xcolor}
\usepackage{graphics}
\usepackage{graphicx}
\usepackage{import}
\usepackage{tikz}
\usepackage{enumitem}
\usetikzlibrary{arrows,automata}
\usepackage[outline]{contour}
\usepackage{blkarray, bigstrut}
\contourlength{1.2pt}

\title{Reciprocal maximum likelihood degrees of diagonal linear concentration models}

 \titlemark{Reciprocal ML degrees of diagonal linear concentration models}

\MSC{14C17, 05B35, 62R01}
\keywords{maximum likelihood degrees, reciprocal spaces, matroids, characteristic polynomials}

\author{Christopher Eur}
\address{%
Department of Mathematics\\
Stanford University\\
\email{chriseur@stanford.edu}
} 

\author{Tara Fife}
\address{%
Max Planck Institute \\ for Mathematics in the Sciences\\
\email{fi.tara@gmail.com}
} 
\authorline

\author{Jos\'e Alejandro Samper}
\address{%
Departamento de Matem\'aticas \\
Pontificia Universidad Cat\'olica de Chile 
\\
\email{jsamper@mat.uc.cl}
} 

\author{Tim Seynnaeve}
\address{%
Mathematical Institute\\
University of Bern\\
\email{tim.seynnaeve@math.unibe.ch}
} 

\authormark{C. Eur\ - \ T. Fife \ -\  J. SAMPER \ -\ T. SEYNNAEVE}

\newcommand{\trace}{\operatorname{trace}}

\newcommand{\ZZ}{\mathbb{Z}}
\newcommand{\CC}{\mathbb{C}}
\newcommand{\RR}{\mathbb{R}}
\newcommand{\PP}{\mathbb {P}}

\newcommand{\cL}{\mathcal{L}}

\newcommand{\mld}{\operatorname{mld}}
\newcommand{\rmld}{\operatorname{rmld}}
\newcommand{\supp}{\operatorname{supp}}
\newcommand{\rk}{\operatorname{rk}}

\begin{document}
\maketitle

\begin{abstract}
We show that the reciprocal maximal likelihood degree (rmld) of a diagonal linear concentration model $\mathcal L \subseteq \CC^n$ of dimension $r$ is equal to
$
(-2)^r\chi_M( \textstyle\frac{1}{2}),
$
where $\chi_M$ is the characteristic polynomial of the matroid $M$ associated to $\mathcal L$.
In particular, this establishes the polynomiality of the rmld for general diagonal linear concentration models, positively answering a question of Sturmfels, Timme, and Zwiernik.
\end{abstract}

\section{Introduction}

Let $\mathbb S^n$ be the space of (real or complex) $n\times n$ symmetric matrices, and $\mathbb S^n_{>0}$ the subset consisting of real positive definite symmetric matrices.  For a fixed $S\in \mathbb S^n_{>0}$, the \textbf{log-likelihood function} $\ell_S: \mathbb S^n_{>0} \to \RR$ is defined by
\[
\ell_S(K) := \log\det K - \trace(S\cdot K).
\]
For a subvariety $\cL \subseteq \mathbb S^n$, the \textbf{maximum likelihood (ML) degree} $\mld(\cL)$ is the number of invertible complex critical points of $\ell_S$ on the smooth locus of $\cL$, counted with multiplicity, for a general choice of $S$.  Writing $\cL^{-1} \subseteq \mathbb S^n$ for the subvariety obtained as the closure of $\{K^{-1} \in \mathbb S^n \mid K\in \cL \text{ invertible}\}$, one defines the \textbf{reciprocal maximum likelihood degree} $\rmld(\cL)$ as the number of invertible complex critical points of $\ell_S$ on the smooth loci of $\cL^{-1}$, counted with multiplicity.

\medskip
Computing (reciprocal) ML degrees arises in statistical applications, where $\mathbb S^n_{>0}$ is often considered as the set of concentration matrices of multivariate normal distributions \cite{SU10}.  We caution that the terminology here regarding reciprocal vs.\ non-reciprocal ML degree is the opposite of that in \cite{STZ20, BCEMR20}, where $\mathbb S^n_{>0}$ is considered as the set of \emph{covariance} matrices (inverses of concentration matrices).  In particular, our $\rmld$ is the ML degree of a linear covariance model.  Our convention here agrees with \cite{SU10, MMW20, FMS20, MMMSV20}. 

\medskip
Let $[n] = \{1, \ldots, n\}$.  A \textbf{diagonal linear concentration model} is a linear subspace $\cL \subseteq \CC^{[n]}$, where $\CC^{[n]}$ is identified with the space of diagonal matrices in $\mathbb S^n$.
Let $M$ be the matroid on $[n]$ whose independent subsets are $I\subseteq [n]$ such that the composition $\CC^I \hookrightarrow \CC^{[n]} \twoheadrightarrow \cL^\vee$ is injective.  Without loss of generality, we always assume that $\cL$ is not contained in a coordinate hyperplane, or equivalently, that $M$ is loopless, since otherwise $\rmld(\cL) = 0$ from the definition.   Our main result is the formula for the reciprocal ML degree of $\cL$ in terms of $M$.

\begin{thm}\label{thm:main1}
Let $\cL \subseteq \CC^{[n]}$ be a diagonal linear concentration model of dimension $r$, and $M$ the associated matroid of rank $r$ on $[n]$.  Then we have
\[
\rmld(\cL) = (-2)^r\chi_M( \textstyle\frac{1}{2}),
\]
where $\chi_M$ is the characteristic polynomial of $M$.
\end{thm}

In \cite{SU10,STZ20}, the (non-reciprocal) ML degree of $\cL$ was shown to be $|\chi_{M}(0)|$.  Computing the reciprocal ML degree presents fundamentally new challenges; see Remark~\ref{rem:MLdeg} for a comparison.

\medskip
From computational experiments, the authors of \cite{STZ20} asked whether the reciprocal ML degree of a \emph{general} diagonal linear concentration model of dimension $r$ in $\CC^{[n]}$ is a polynomial in $n$ of degree $r-1$.
Evaluating our Theorem~\ref{thm:main1} at uniform matroids answers their question positively.

\begin{cor}\label{cor:uniform}
Let $\cL \subseteq \CC^{[n]}$ be a general linear concentration model of dimension $r$.
Then we have
\[
\rmld(\cL) =  \sum_{i=1}^{r}  \textstyle \binom{n-i-1}{r-i}2^{r-i}.
\]
\end{cor}

For instance, when $r = 3$ we have $2n^2 - 8n + 7$, and when $r = 4$ we have $4/3n^3 - 10n^2 + 68/3n - 15$, as predicted in \cite{STZ20} from numerical computations.

\medskip
To prove Theorem~\ref{thm:main1}, we use the following alternate description of the reciprocal ML degree, obtained by a standard computation in multivariable calculus.  Let $\cL^\perp$ denote the orthogonal complement of a subspace $\cL \subseteq \CC^{[n]}$ under the standard pairing
\[
\langle (x_1, \ldots, x_n), (y_1, \ldots, y_n) \rangle := \sum_{i=1}^n x_i y_i.
\]

\begin{prop}\label{prop:scoreprep}\cite[Proposition 4.3]{STZ20}
The reciprocal ML degree of a linear subspace $\cL \subseteq \CC^{[n]}$ is equal to the number of solutions $(x_1,\ldots,x_n) \in (\CC^*)^{[n]}$, counted with multiplicity, to the following system of equations, where $s_1,\ldots,s_n \in \CC$ are generic parameters:
\[
	(x_1^{-1},\ldots,x_n^{-1}) \in \cL \quad\text{and}\quad
	(s_1x_1^2-x_1,\ldots,s_nx_n^2-x_n) \in \cL^{\perp}.
\]
\end{prop}

Thus, we prove Theorem~\ref{thm:main1} by establishing the following generalization.

\begin{thm}\label{thm:main2}
For an $r$-dimensional linear subspace $\cL\subseteq \CC^{[n]}$, a generic choice of parameters $s_1, \ldots, s_n \in \CC$, and any integer $d\geq 1$, the number of solutions $(x_1, \ldots, x_n) \in (\CC^*)^{[n]}$, counted with multiplicity, to the system of equations
\begin{equation}\label{eq:score}\tag{$\dagger$}
(x_1^{-1},\ldots,x_n^{-1}) \in \cL \quad\text{and}\quad (s_1x_1^d-x_1,\ldots,s_nx_n^d-x_n) \in \cL^{\perp}
\end{equation}
is equal to
\[
(-d)^r\chi_M( \textstyle\frac{1}{d}),\quad \text{or equivalently,} \quad d^r T_M(1- \frac{1}{d},0),
\]
where $\chi_M$ is the characteristic polynomial and $T_M$ is the Tutte polynomial of the matroid $M$ associated to $\cL$.
\end{thm}

\begin{rem}
For an $r$-dimensional subspace $\cL \subseteq \CC^{[n]}$, let $U(\cL) := \cL \cap (\CC^*)^{[n]}$ be the hyperplane arrangement complement, and $M$ the associated matroid.  Then the Poincar\'e polynomial of $U(\cL)$
\[
P_{U(\cL)}(q) := \sum_{i\geq 0} \big(\operatorname{rank} H_i(U(\cL);\ZZ)\big)q^i
\]
coincides with polynomial $(-q)^r\chi_M(-\frac{1}{q})$ \cite{OS80}.  In particular, Theorem~\ref{thm:main1} states that $\rmld(\cL) = (-1)^rP_{U(\cL)}(-2)$.  This echoes the result of \cite{Huh13}, which showed that, for a different log-likelihood function (from discrete statistical models), the ML degree of a smooth very affine variety $U$ is its signed topological Euler characteristic $(-1)^{\dim U}P_U(-1)$.  However, ML degrees in our case are not topological invariants of very affine varieties: Observe that $\cL^{-1}\cap (\CC^*)^{[n]} \simeq \cL\cap (\CC^*)^{[n]}$ but in general $\rmld(\cL^{-1}) = \mld(\cL) = \chi_M(0) \neq (-2)^r\chi_M(\frac{1}{2}) = \rmld(\cL)$.
It may still be interesting to find other families of subvarieties $\cL\subseteq \mathbb S^n$ such that $\rmld(\cL) = (-1)^{\dim \cL}P_{U(\cL)}(-2)$, where $U(\cL) := \{K\in \cL \mid K \text{ invertible}\}$.  For example, general pencils of conics form one such family \cite{CMR20, FMS20}.
\end{rem}

\noindent\textbf{Outline.}
In Section~\ref{sec:reclinsp} we review properties of reciprocal linear spaces $\cL^{-1}$, and introduce score varieties, which together with $\cL^{-1}$ encode the system of equations \eqref{eq:score}.  After establishing two key technical lemmas in Section~\ref{sec:keylemmas}, in Section~\ref{sec:total} we compute the number of solutions to the system of equations \eqref{eq:score} in $\CC^{[n]}$, instead of in $(\CC^*)^{[n]}$, in two different ways: One is a B\'ezout-like computation, and the other is a summation, with each summand corresponding to a set of solutions with specified support (non-zero coordinates).  An inclusion-exclusion argument in Section~\ref{sec:comb} then yields the proof of Theorem~\ref{thm:main2}.

\subsection*{Notation}
For an affine subvariety $X \subseteq \CC^n$, we write $\overline X \subseteq \PP^n$ for its projective closure.  If $X\subseteq \CC^n$ is defined by a homogeneous ideal, then we write $\PP X \subseteq \PP^{n-1}$ for its projectivization.  For a point $p\in X\subseteq \CC^{[n]}$, we write $TC_pX \subseteq \CC^{[n]}$ for the tangent cone of $X$ at $p$.  For $p = (p_1, \ldots, p_n) \in \CC^{[n]}$, we write $\supp(p) := \{i\in [n] \mid p_i \neq 0\}$ for its support, and for $I \subseteq [n]$, write $p_{|I}$ for the projection of $p$ onto $\CC^I \subseteq \CC^{[n]}$.

\section{Reciprocal linear spaces and score varieties}\label{sec:reclinsp}

We set notations concerning matroids associated to linear subspaces, and review necessary facts about reciprocal linear spaces.  We assume familiarity with matroid theory, and refer to \cite{Wel76, Oxl11} as standard references.

\medskip
Let us fix a linear subspace $\cL \subseteq \CC^{[n]}$ of dimension $r$.  Let $A$ be an $r\times n$ matrix whose row-span equals $\cL$.
We will often use the fact that the minimal sets among supports of elements in the row-span of $A$ form the cocircuits of $M$.

\medskip
For a subset $I\subseteq [n]$, let $\cL_{|I} \subseteq \CC^I$ be the image of $\cL$ under the coordinate projection $\CC^{[n]} \twoheadrightarrow \CC^I$, and let $\cL_{/I}$ be the intersection of $\cL$ with the coordinate subspace $\{0\}^I \times \CC^{[n]\setminus I}$, considered as a subspace of $ \CC^{[n]\setminus I}$.  The matroid of $\cL_{|I}$ is the restriction $M|I$, 
whereas the matroid of $\cL_{/I}$ is the contraction $M/I$.

\medskip
The \textbf{reciprocal linear space} $\cL^{-1}$ of $\cL$ is the Zariski closure in $\CC^{[n]}$ of $\{(x_1, \ldots, x_n) \in (\CC^*)^{[n]} \mid (x_1^{-1}, \ldots, x_n^{-1})\in \cL\}$.  Note that $\cL^{-1} \cap (\CC^*)^{[n]}$ is smooth, being isomorphic to $\cL \cap (\CC^*)^{[n]}$.  For $I\subseteq [n]$, we write $\cL^{-1}_{|I}$ for $(\cL_{|I})^{-1}$, and likewise write $\cL^{-1}_{/I} = (\cL_{/I})^{-1}$. We collect together in the following theorem the known properties of $\cL^{-1}$ that we will need.

\begin{thm}\label{thm:reclinsp}
Let $\cL^{-1} \subseteq \CC^{[n]}$ be the reciprocal linear space of $\cL \subseteq \CC^{[n]}$.  
\begin{enumerate}[label=(\alph*)]
\item \label{deg} \cite[Lemma 2]{PS06} The ideal of $\cL^{-1}$ is homogeneous, and $\PP \cL^{-1}$ has degree $|\mu(M)|$, where $\mu(M) := \chi_M(0)$ is the M\"obius invariant of the matroid $M$.
\item \label{CM} \cite[Proposition 7]{PS06} The ideal of $\cL^{-1}$ is Cohen-Macaulay, with any basis of $\cL^\perp$ forming a system of parameters, i.e.\ $\cL^{-1} \cap \cL^\perp = \{\mathbf 0\}$.
\item \label{strata} \cite[Proposition 5]{PS06} The intersection $\cL^{-1} \cap ((\CC^*)^F \times \{0\}^{[n]\setminus F})$ is nonempty if and only if $F\subseteq [n]$ is a flat of $M$, and in that case, one has
\[
\cL^{-1} \cap ((\CC^*)^F \times \{0\}^{[n]\setminus F}) = (\cL_{|F}^{-1} \cap (\CC^*)^F) \times \{0\}^{[n]\setminus F}.
\]
\item \label{TC} \cite[Theorem 24]{SSV13} For a flat $F\subseteq [n]$ %
and a point $p\in \cL^{-1}$ with $\supp(p) = F$,
the tangent cone of $\cL^{-1}$ at $p$ is the product
\[
TC_p\cL^{-1} = TC_{p_{|F}}\cL_{|F}^{-1} \times \cL_{/F}^{-1} \simeq \cL_{|F} \times \cL_{/F}^{-1}.
\]
\end{enumerate}
\end{thm}

All four statements in Theorem~\ref{thm:reclinsp} can be derived easily from the Gr\"obner basis for the defining ideal of $\cL^{-1}$ computed in \cite[Theorem 4]{PS06}.
In \cite{SSV13}, the statement of Theorem~\ref{thm:reclinsp}.\ref{TC} originally reads $TC_p\cL^{-1} = p_{|F}^2 \cL_{|F} \times \cL_{/F}^{-1}$, where $p^2 \cL$ denotes the linear subspace $\{ (p_1^2 x_1, \ldots, p_n^2 x_n) \mid (x_1, \ldots, x_n) \in \cL\}$.  It is straightforward to verify that $TC_{p_{|F}}\cL_{|F}^{-1} = p_{|F}^2 \cL_{|F}$.  Theorem~\ref{thm:reclinsp}.\ref{deg} also follows from \cite[Theorem 1.2]{Ter02}, which expressed the Hilbert series of the ideal of $\cL^{-1}$ in terms of the characteristic polynomial $\chi_M$.

\medskip
The reciprocal linear space $\cL^{-1}$ encodes the left half of the system of equations in Equation~\eqref{eq:score}.
Let us now consider the variety encoding the condition $(s_1x_1^d - x_1, \ldots, s_nx_n^d - x_n) \in \cL^\perp$.  For an integer $d\geq 1$ and a parameter $s = (s_1, \ldots, s_n) \in \CC^n$, we define the \textbf{score variety} as
\[
Y(\cL, s, d) := \{ (x_1, \ldots, x_n) \in \CC^{[n]} \mid (s_1x_1^d-x_1, \ldots, s_nx_n^d - x_n) \in \cL^\perp\}.
\]
We will simply write $Y$ when we trust that no confusion will arise.  We note here that score varieties are smooth for a generic choice of $s\in \CC^n$.

\begin{lemma}\label{lem:scoresmooth}
For $d \geq 1$ and a generic choice of $(s_1, \ldots, s_n)\in \CC^n$, the score variety $Y$ is smooth.
\end{lemma}

\begin{proof}
If $d=1$, then $Y$ is linear, so suppose $d\geq 2$.  Let $A$ be the $r\times n$ matrix whose row span equals $\cL$, and let $g_1, \ldots, g_r$ be the polynomials obtained by multiplying the rows of $A$ with $(s_1x_1^d - x_1, \ldots, s_nx_n^d - x_n)^T$.  These minimally generate the defining ideal $I_Y\subseteq \CC[x_1, \ldots, x_n]$ of $Y$.  The Jacobian matrix with respect to these minimal generators is
 \begin{equation}\label{eq:jacobian}
\operatorname{Jac}(x) = A \cdot \operatorname{diag}(ds_1x_1^{d-1} - 1, \ldots, ds_nx_n^{d-1} -1),
\end{equation}
i.e.\ matrix $A$ whose $i$-th column is scaled by $ds_ix_i^{d-1} -1$ for each $1\leq i \leq n$.
Suppose now that $\operatorname{Jac}(x)$ has rank $< r$ for some $x\in \CC^n$, that is, the restriction $M|I$ of the matroid $M$ to the set
$
I = \{i\in[n] \mid ds_ix_i^{d-1} - 1 \neq 0\}
$
has rank $<r$.  This happens if and only if $I$ is contained in a hyperplane flat of $M$, or equivalently, the subset $J := [n]\setminus I$ contains a cocircuit of $M$.
As the minimal supports of the row-space of $A$ constitute the cocircuits of $M$, let $v = (v_1, \ldots, v_n) \in \CC^n$ be the element in the row-space of $A$ whose support $C^* = \supp(v) \subseteq [n]$ is a cocircuit of $M$ contained in $J$.  Then we have
\[
v \cdot (s_1x_1^d - x_1, \ldots, s_nx_n^d - x_n)^T = \sum_{i\in C^*} v_i (s_ix_i^d - x_i) = (\textstyle{\frac{1}{d}} - 1) \displaystyle \sum_{i\in C^*} v_i  x_i
\]
where last equality follows from $ds_ix_i^{d-1} -1 = 0$ for $i\in J$.  This quantity needs to be zero if $x\in Y$.  We claim that for a general choice of $(s_1, \ldots, s_n)$ this quantity can never be zero:  Consider the set
\[
Z := \{ (\zeta_1, \ldots, \zeta_n) \in \CC^n \mid \zeta_i \text{ is a $(d-1)$-th root of } \textstyle{\frac{1}{ds_i}} \text{ if $i\in C^*$}\}.
\]
For a general choice of $(s_1,\ldots,s_n)$, no element of $Z$ satisfies $\sum_{i\in C^*} v_ix_i=0$.
\end{proof}

\section{Two genericity lemmas}\label{sec:keylemmas}

We now present the two key technical lemmas for our future intersection multiplicity computations.  Both make essential use of the fact that the parameter $s\in \CC^n$ can be chosen generically, and the second lemma uses that $\CC$ has characteristic zero.  To state the first lemma, let us define a subscheme of $\CC^{[n]}$
\[
Y_\infty(\cL,s,d) := \{(x_1, \ldots, x_n)\in \CC^{[n]} \mid (s_1x_1^d, \ldots, s_nx_n^d) \in \cL^\perp\}.
\]

\begin{lemma}\label{lem:intersectionAtInfinity}
For a generic choice of $s\in \CC^n$ and any integer $d\geq 1$, one has
\[
\cL^{-1}\cap Y_\infty(\cL,s,d) = \{\mathbf 0\}.
\]
\end{lemma}

\begin{proof}
Let us define a subscheme $V \subseteq \CC^{[n]}\times \CC^{[n]}$ by
\[
V = \{ (x,s) \in \CC^{[n]}\times \CC^{[n]} \mid x \in \cL^{-1} \text{ and } (s_1x_1^d, \ldots, s_nx_n^d) \in \cL^\perp\}.
\]
For any point $x$ in the dense open loci $\cL^{-1}\cap (\CC^*)^{[n]}$ of $\cL^{-1}$, the set $\{s \in \CC^{[n]}\mid (s_1x_1^d, \ldots, s_nx_n^d) \in \cL^\perp\}$ is a linear subspace of dimension $n-r$, so the dimension of $V$ is $r + (n-r) = n$.  Moreover, the subscheme $V$ is bi-homogeneous, and thus the bi-projectivization $\overline{V} \subseteq \PP^{n-1} \times \PP^{n-1}$ has dimension $n-2$.  Writing $\pi_2$ for the projection of $\overline{V}$ to the second $\PP^{n-1}$, we hence find that the loci $\PP^{n-1}\setminus \pi_2(\overline{V})$ is nonempty and open in $\PP^{n-1}$.  That is, the affine cone over $\PP^{n-1}\setminus \pi_2(\overline{V})$ is dense open in $\CC^{[n]}$, and $\cL^{-1} \cap Y_\infty(\cL,s,d) = \{\mathbf 0\}$ for any $s\in \CC^{[n]}$ in the affine cone.  
\end{proof}

\begin{rem}
When $s = (1,\ldots, 1)$ and $d=1$, Lemma~\ref{lem:intersectionAtInfinity} is the second half of Theorem~\ref{thm:reclinsp}.\ref{CM}, which was established by an explicit Gr\"obner basis computation.  For $d\geq 2$ however, the lemma fails in general with $s = (1, \ldots, 1)$.  
\end{rem}

\begin{lemma}\label{lem:final}
For a generic choice of $s\in \CC^n$ and any integer $d\geq 1$, the intersection $\cL^{-1} \cap Y(\cL, s, d) \cap (\CC^*)^{[n]}$ is either empty or smooth of dimension 0.
\end{lemma}

\begin{proof}
Without loss of generality, we assume that the $r\times n$ matrix $A$ whose row-span equals $\cL$ is of the form $[I_r \ | \  A']$, where $I_r$ is the $r\times r$ identity matrix.  For $1\leq i \leq r$, let $a'_i$ be the $i$-th row of $A'$.  The ideal of $Y$ is minimally generated by
\[
\mathfrak d = \{(s_ix_i^d - x_i) - a'_i \cdot (s_{r+1}x_{r+1}^d-x_{r+1}, \ldots, s_nx_n^d-x_n)^T\mid 1\leq i \leq r\}.
\]
Fixing a generic choice of $s_{r+1}, \ldots, s_n$, and letting $s_1, \ldots, s_r$ vary freely, for each $i = 1, \ldots, r$ we may consider $(s_ix_i^d - x_i) - a'_i \cdot (s_{r+1}x_{r+1}^d-x_{r+1}, \ldots, s_nx_n^d-x_n)^T$ as a pencil $\mathfrak d_i$ of hypersurfaces in $\CC^{[n]}$.  Note that the union of the base loci of $\mathfrak d_1, \ldots, \mathfrak d_r$ is contained in the union of the coordinate hyperplanes.  Thus, we obtain a map $\cL^{-1}\cap (\CC^*)^{[n]} \to (\PP^1)^r$ of smooth varieties.  By generic smoothness \cite[III.10.7]{Har77}, the general fiber, which is the intersection $\cL^{-1} \cap Y \cap (\CC^*)^{[n]}$ for a general choice of $s\in \CC^n$, is either empty or smooth of dimension 0.
\end{proof}

Let us now denote 
\[
\begin{split}
\mathcal D(\cL,d) :=& \text{ the degree of the (empty or 0-dimensional) subscheme} \\
&\ \cL^{-1} \cap Y(\cL,s,d) \cap (\CC^*)^{[n]} \subset \CC^{[n]}
\end{split}
\]
for a generic choice of $s\in \CC^n$, which is equal to the number of points in the intersection since it is smooth by Lemma~\ref{lem:final}.  Theorem~\ref{thm:main2} is now equivalently stated as $\mathcal D(\cL,d) = (-d)^r\chi_M(\frac{1}{d})$, and Proposition~\ref{prop:scoreprep} states that $\mathcal D(\cL,2) = \rmld(\cL)$.

\section{Total intersection multiplicity}\label{sec:total}

We now compute the degree of the intersection $\cL^{-1} \cap Y(\cL,s,d)$ as a subscheme of $\CC^{[n]}$ in two different ways.  First, we have a B\'ezout-like computation.

\begin{prop}\label{prop:total}
For a generic choice of $s\in \CC^n$, the intersection $\cL^{-1} \cap Y(\cL,s,d)$ is a 0-dimensional scheme of degree $d^r|\mu(M)|$.
\end{prop}

\begin{proof}
For $i = 1, \ldots, r$, let $\overline f_i \in \CC[x_0, x_1, \ldots, x_n]$ be the homogeneous polynomial obtained as $i$-th row of $A$ times $(s_1x_1^d - x_1x_0^{d-1}, \ldots, s_nx_n^d - x_nx_0^{d-1})^T$.  We first claim that $(\overline f_1, \ldots, \overline f_r)$ forms a regular sequence on the projective closure $\overline {\cL^{-1}} \subset \PP^n$.  
Since the projective variety $\overline {\cL^{-1}}$ is (arithmetically) Cohen-Macaulay by Theorem~\ref{thm:reclinsp}.\ref{CM}, it suffices to show that the intersection $\overline{\cL^{-1}} \cap V(\overline f_1, \ldots, \overline f_r)$ is 0-dimensional, as every system of parameters in a standard graded Cohen-Macaulay ring is a regular sequence \cite[Theorem 2.1.2]{BH93}.

At the hyperplane at infinity, the intersection $V(x_0) \cap \overline{\cL^{-1}} \cap V(\overline f_1, \ldots, \overline f_r)$ is isomorphic to $\PP\cL^{-1} \cap \PP Y_\infty$, which is empty for a generic $s\in \CC^n$ by Lemma~\ref{lem:intersectionAtInfinity}.  On the complement of the hyperplane at infinity, the intersection is equal to $\cL^{-1} \cap Y$, since the dehomogenizations of the polynomials $\overline f_1, \ldots, \overline f_r$ give the defining equations of $Y$.  From Theorem~\ref{thm:reclinsp}.\ref{strata} and the definition of $Y(\cL,s,d)$, it follows that
\[
\cL^{-1}\cap Y \cap ((\CC^*)^F \times \{0\}^{[n]\setminus F}) = (\cL_{|F}^{-1} \cap Y(\cL_{|F},s_{|F},d) \cap (\CC^*)^F ) \times \{0\}^{[n]\setminus F}
\]
if $F\subseteq [n]$ is a flat or empty otherwise, and thus Lemma~\ref{lem:final} applied to each flat $F$ implies that $\cL^{-1} \cap Y$ is 0-dimensional.  Thus, the degree $d$ polynomials $(\overline f_1, \ldots, \overline f_r)$ form a regular sequence on $\overline {\cL^{-1}}$, and hence the degree of $\overline{\cL^{-1}} \cap V(\overline f_1, \ldots, \overline f_r)$ is $d^r \deg (\overline{\cL^{-1}})$.
As the ideal of $\cL^{-1}$ homogeneous, the degrees of $\PP\cL^{-1}$ and $\overline{\cL^{-1}}$ are equal, with the value being $|\mu(M)|$ by Theorem~\ref{thm:reclinsp}.\ref{deg}.  Lastly, since $\overline{\cL^{-1}} \cap V(\overline f_1, \ldots, \overline f_r)$ is empty at the hyperplane at infinity, the degree of the intersection is equal to the degree of $\cL^{-1} \cap Y(\cL,s,d)$.
\end{proof}

We now compute the degree of $\cL^{-1} \cap Y(\cL,s,d)$ as the sum of contributions from the various strata of $\cL^{-1}$.  First, we need the following notation.  Note that a 0-dimensional subscheme $X \subset \CC^{[n]}$ is a union $\bigcup_{\alpha} X_\alpha$ of irreducible (possibly non-reduced) components $X_\alpha$, each of which is topologically a point $(X_\alpha)_{red}$ in $\CC^{[n]}$.  For a subset $I\subseteq [n]$, we write $X^F$ to be the subscheme of $X$ defined as the union of components of $X$ whose support is $F$, i.e.\
\[
X^F := \bigcup \{X_\alpha \mid (X_\alpha)_{red} \in (\CC^*)^F \times \{0\}^{[n]\setminus F}\}.
\]
Moreover, recall the notation that $\mathcal D(\cL,d)$ denotes the degree of the subscheme $\cL^{-1} \cap Y(\cL,s,d) \cap (\CC^*)^{[n]} \subset \CC^{[n]}$ for a generic choice of $s\in \CC^n$.

\begin{prop}\label{prop:eachstrata}
For a generic choice of $s\in \CC^n$, and for a flat $F\subseteq [n]$ of $M$, the degree of $(\cL^{-1} \cap Y(\cL,s,d))^F \subset \CC^{[n]}$ is equal to $\mathcal D(\cL_{|F}, d) \cdot |\mu(M/F)|$.
\end{prop}

\begin{proof}
As topological spaces, the subscheme $(\cL^{-1} \cap Y)^F$ is equal to the intersection $\cL^{-1} \cap Y \cap ((\CC^*)^F \times \{0\}^{[n]\setminus F})$, which is by Theorem~\ref{thm:reclinsp}.\ref{strata} isomorphic to $\cL_{|F}^{-1} \cap Y(\cL_{|F}, s_{|F}, d) \cap (\CC^*)^F$, which as a scheme is a disjoint union of $\mathcal D(\cL_{|F},d)$ many smooth points by Lemma~\ref{lem:final}.  It remains only to show that if $\widetilde p$ is an irreducible component in $(\cL^{-1} \cap Y)^F$, then the degree of $\widetilde p$ is $|\mu(M/F)|$.

For this end, we recall \cite[Theorem 1.26 \& Proposition 1.29]{EH16}:  Suppose two Cohen-Macaulay subvarieties $X$ and $X'$ of complementary dimensions in a smooth variety $Z$ intersect dimensionally properly.  Then, the degree of the intersection $X\cap X'$ at a point $q$ in the intersection is equal to the product of the degrees of projectivized tangent cones $\PP TC_qX$ and $\PP TC_qX'$, provided that $\PP TC_qX$ and $\PP TC_qX'$ are disjoint in $\PP T_qZ$.

We apply this to $\cL^{-1}$ and $Y$, which are Cohen-Macaulay respectively by Theorem~\ref{thm:reclinsp}.\ref{CM} and Lemma~\ref{lem:scoresmooth}, and they intersect dimensionally properly by Proposition~\ref{prop:total}.  Topologically $\widetilde p$ is a point $p\in (\CC^*)^F \times \{0\}^{[n]\setminus F}$.  Combining Theorem~\ref{thm:reclinsp}.\ref{TC} with Theorem~\ref{thm:reclinsp}.\ref{deg}, one has that the degree of $\PP TC_p\cL^{-1}$ is equal to $|\mu(M/F)|$.  The degree of $\PP TC_pY$ is 1 since $Y$ is smooth (Lemma~\ref{lem:scoresmooth}).  Thus, we are done once we show that $\PP TC_p\cL^{-1}$ and $\PP TC_pY$ are disjoint.
This is done in the following lemma.
\end{proof}

\begin{lemma}
Let $p$ be a point in $\cL^{-1} \cap Y \cap ((\CC^*)^F\times \{0\}^{[n]\setminus F})$ for a generic choice of $s\in \CC^n$, and for a flat $F\subseteq [n]$.  Then we have $\PP TC_p\cL^{-1} \cap \PP TC_pY = \emptyset$.
\end{lemma}

\begin{proof}
Let $A$ be an $r\times n$ matrix whose row-span is $\cL$.  As computed in the proof of Lemma~\ref{lem:scoresmooth} in Equation~\eqref{eq:jacobian}, the tangent cone $TC_pY$ is equal to
\[
\ker \big(A \cdot \operatorname{diag}(ds_1p_1^{d-1}-1, \ldots, ds_np_n^{d-1}-1)\big) \subset \CC^{[n]}.
\]
Since $p_i = 0$ for $i\in [n]\setminus F$, and since the cocircuits of the matroid $M/F$ are cocircuits of $M$ contained in $[n]\setminus F$, if $x\in TC_pY$ then $x_{|[n]\setminus F} \in \cL_{/F}^\perp \subset \CC^{[n]\setminus F}$.
On the other hand, by Theorem~\ref{thm:reclinsp}.\ref{TC} we have
$
TC_p\cL^{-1} = TC_{p_{|F}}\cL_{|F}^{-1} \times \cL_{/F}^{-1}.
$
Let us now consider $x \in TC_p\cL^{-1}\cap TC_pY$.   We have $x_{|[n]\setminus F} \in \cL_{/F}^\perp \cap \cL_{/F}^{-1} = \{\mathbf 0\}$, where the equality follows from Theorem~\ref{thm:reclinsp}.\ref{CM}, and thus $x = x_{|F}\times \mathbf 0$.  That $x\in TC_pY$ now implies that $x_{|F} \in TC_{p_{|F}}Y(\cL_{|F},s_{|F},d)$.  Thus, we conclude $x_{|F}=0$, since by Lemma~\ref{lem:final} the intersection $\cL_{|F}^{-1} \cap Y(\cL_{|F},s_{|F},d)$ is smooth, and in particular transversal at $p_{|F}$, i.e.\ $TC_{p_{|F}}\cL_{|F}^{-1} \cap TC_{p_{|F}}Y(\cL_{|F},s_{|F},d) = \{\mathbf 0\}$.
\end{proof}

Combining Propositions \ref{prop:total} and \ref{prop:eachstrata} yields the following.

\begin{cor}\label{cor:twocounts}
For a generic choice of $s\in \CC^n$, we have
\[
d^r|\mu(M)| = \deg\big( \cL^{-1} \cap Y(\cL,s,d)\big) = \sum_{\substack{F\subseteq[n]\\ \text{a flat of $M$}}} \mathcal D(\cL_{|F},d) |\mu(M/F)|.
\]
\end{cor}

\section{Inclusion-exclusion}\label{sec:comb}

We now finish the proof of the main theorem by combining Corollary~\ref{cor:twocounts} with an inclusion-exclusion argument.  For the facts regarding the lattice of flats of a matroid and the M\"obius invariant used here, see \cite{Zas87}.

\begin{proof}[Proof of Theorem~\ref{thm:main2}]
Write $\operatorname{rk}: 2^{[n]} \to \ZZ$ for the rank function of the matroid $M$ associated to $\cL$.  Let us recall that for a matroid $M'$ of rank $r'$, one has $(-1)^{r'} \mu(M') = |\mu(M')|$.
Then, Corollary~\ref{cor:twocounts} states that
\[
(-d)^{\operatorname{rk} [n]} \mu(M) = \sum_{\substack{F\subseteq[n]\\ \text{a flat of $M$}}} (-1)^{\operatorname{rk} [n] - \operatorname{rk} F}\mathcal D(\cL_{|F},d) \mu(M/F),
\]
or more generally, one has, for any flat $F\subseteq [n]$ of $M$,
\[
(-d)^{\operatorname{rk}F} \mu(M|F) = \sum_{\substack{F'\subseteq F\\ \text{a flat of $M$}}} (-1)^{\operatorname{rk} F - \operatorname{rk} F'}\mathcal D(\cL_{|F'},d) \mu(M|F/F').
\]
As $\mu(M|F/F')$ is the value of the M\"obius function $\mu(F',F)$ on the lattice of flats of $M$, applying the M\"obius inversion formula \cite[Proposition 3.7.1]{Sta12} (with $f(F)=d^{\rk F}\mu({M}|{F})$ and $g(F)=(-1)^{\rk F}\mathcal D(\cL_{|F},d)$) yields 
\[
(-1)^{\rk F}\mathcal D(\cL_{|F},d) = \sum_{\substack{F'\subseteq F\\ \text{a flat of $M$}}}d^{\rk F'}\mu({M}|{F'}).
\]
Now, letting $F = [n]$ and noting $\operatorname{rk} [n] = r$, we have
\[
(-1)^{r} \mathcal D(\cL,d) = d^{r} \sum_{\substack{F'\subseteq [n]\\ \text{a flat of $M$}}} ( \textstyle{\frac{1}{d}})^{r - \rk F'}\mu(M|F') = 
d^r \chi_M(\frac{1}{d}),
\]
so that $\mathcal D(\mathcal L,d) = (-d)^r\chi_M(\frac{1}{d}) = d^rT_M(1-\frac{1}{d},0)$ as desired.
\end{proof}

\begin{rem}\label{rem:MLdeg}
For the non-reciprocal ML degree $\mld(\cL)$, a standard computation similar to the one that gives Proposition~\ref{prop:scoreprep} (see \cite{SU10, STZ20}) yields
\[
\mld(\cL) = \deg \big(\cL^{-1} \cap Y(\cL,s,d=0) \cap (\CC^*)^{[n]} \big) = \mathcal D(\cL, d=0).
\]
One can hence recover \cite[Corollary 3]{SU10}, which states $\mld(\cL) = |\mu(M)|$, by minor modifications of our arguments here.  This case is in fact simpler, with no need for the consideration of tangent cones as was done in Proposition~\ref{prop:eachstrata}, because the intersection $\cL^{-1} \cap Y(\cL,s,d)$ lies entirely in $(\CC^*)^{[n]}$ when $d=0$.  We emphasize that $\cL^{-1} \cap Y(\cL,s,d)$ \emph{never} lies entirely in $(\CC^*)^{[n]}$ when $d\geq 1$.
\end{rem}

\begin{rem}
Combining Theorem~\ref{thm:main2} with the ``recipe formula'' for Tutte-Grothendieck invariants of matroids (see for instance \cite[Theorem 2.16]{Wel99}), one has that $\mathcal D(\cL,d)$ for $d\geq 1$ satisfies the following deletion-contraction relation given an element $e\in [n]$:
\[
\mathcal D(\cL,d) = 
\begin{cases}
0 &\text{if $e$ is a loop,}\\
(d-1) \cdot \mathcal D(\cL_{/e},d) & \text{if $e$ is a coloop,}\\
\mathcal D(\cL_{\setminus e},d) + d\cdot \mathcal D(\cL_{/e},d) & \text{if $e$ neither loop nor coloop,}
\end{cases}
\]
where the base cases are $\mathcal D(\cL= \CC^{[1]},d) = 1$ and $\mathcal D(\cL = \{0\} \subset \CC^{[1]},d) = 0$.
From this deletion-contraction relation, one can verify that the following statements are equivalent:
\begin{enumerate}[label = (\roman*)]
\item $\mathcal D(\cL,2) = \rmld(\cL) = 1$,
\item $M$ is a partition matroid (i.e.\ every component of $M$ has rank 1),
\item $\mld(\cL) = |\mu(M)| = 1$, and
\item $\cL^{-1}$ is linear.
\end{enumerate}
It may be interesting to find a proof of Theorem~\ref{thm:main2} that directly reflects the deletion-contraction relation above geometrically.
\end{rem}

\begin{proof}[Proof of Corollary~\ref{cor:uniform}]
For an $r$-dimensional general subspace $\cL\subseteq \CC^{[n]}$, the associated matroid is the uniform matroid $U_{r,n}$, for which the Tutte polynomial has the following formula (see for instance \cite{MR})
\[
T_{U_{r,n}}(x,y) = \sum_{i=1}^{r}\textstyle \binom{n-i-1}{r-i} x^i + \displaystyle \sum_{j=1}^{n-r}  \textstyle \binom{n-j-1}{r-1} y^j.
\]
Thus, for a general $\cL$, Theorem~\ref{thm:main2} implies that
\[
D(\cL,d) = d^r T_{U_{r,n}}\left( \textstyle{1-\frac{1}{d}},0\right) =   \displaystyle d^r\sum_{1=1}^{r}\textstyle \binom{n-i-1}{r-i}\left(1-\frac{1}{d} \right)^{i}.
\]
Evaluating at $d=2$ gives us
\[
 \rmld(\cL) = 2^r\sum_{i=1}^{r}  \textstyle \binom{n-i-1} {r-i}\left(\frac{1}{2} \right)^{i} \displaystyle  = \sum_{i=1}^{r}  \textstyle \binom{n-i-1}{r-i}2^{r-i},
\]
as desired.
\end{proof}

\begin{rem}
It is interesting to note that, for diagonal linear concentration models, our formula Theorem~\ref{thm:main1} implies that $\rmld(\cL)$ is always odd unless it is zero.  The same pattern seems to persist for reciprocal ML degrees of general linear concentration models \cite[Table 1]{STZ20}.
\end{rem}

\subsection*{Acknowledgements}
We thank Yairon Cid Ruiz and Bernd Sturmfels for helpful conversations, and we thank the referee for simplifying the proof of Lemma 3.1. We also thank the organizers of the Linear Spaces of Symmetric Matrices working group at MPI MiS Leipzig.  C.E. is partially supported by the US National Science Foundation
(DMS-2001854).
 
\bibliography{paper}

\end{document}